\documentclass[12pt,twoside]{article}
\usepackage{amsmath, amsfonts, amsthm, amssymb, amscd,mathrsfs}

\usepackage{bbm}

\usepackage{hyperref}
\makeatletter
\newdimen\rh@wd
\newdimen\rh@hta
\newdimen\rh@htb
\newbox\rh@box
\def\rh@measure#1{\setbox\rh@box=\hbox{$#1$}\rh@wd=\wd\rh@box \rh@hta=\ht\rh@box}

\def\widecheck#1{\rh@measure{#1}%
  \setbox\rh@box=\hbox{$\widehat{\vrule height \rh@hta width\z@ \kern\rh@wd}$}%
  \rh@htb=\ht\rh@box \advance\rh@htb\rh@hta \advance\rh@htb\p@
  \ooalign{$\vrule height \ht\rh@box width\z@ #1$\cr
           \raise\rh@htb\hbox{\scalebox{1}[-1]{\box\rh@box}}\cr}}
\makeatother

\usepackage[most]{tcolorbox} 

\newtcolorbox{markbox}{%
     enhanced, breakable, size=minimal, parbox=false, after={\par}, 
     before upper={\indent}, colback=white, 
     overlay = {\draw[line width=2pt] (frame.north east) -|
                       ([xshift=3mm]frame.east)|-(frame.south east);},
     overlay first={\draw[line width=2pt] (frame.north east) -|
                           ([xshift=3mm]frame.south east);},
     overlay middle={\draw[line width=2pt] ([xshift=3mm]frame.north east) -- 
                              ([xshift=3mm]frame.south east);},
     overlay last={\draw[line width=2pt] ([xshift=3mm]frame.north east)|-
                          (frame.south east);},
}
  
\numberwithin{equation}{section}  
\usepackage{perpage} 
\MakePerPage{footnote}  
\oddsidemargin 	.in
\evensidemargin.in
\topmargin    - 0.75in
\numberwithin{equation}{section}
        \newtheorem{theorem}{Theorem}[section]
        \newtheorem{proposition}[theorem]{Proposition}
        \newtheorem{lemma}[theorem]{Lemma}

        \newtheorem{remark}[theorem]{Remark}  
\let\oldmarginpar\marginpar
\renewcommand\marginpar[1]{\-\oldmarginpar[\raggedleft\footnotesize #1]
{\raggedright\footnotesize #1}}

\newcommand \bei {\begin{itemize}}

\newcommand \eei {\end{itemize}}

\newcommand \fh {\widehat{f}}

\newcommand \be {\begin{equation}}
\newcommand \bel {\be\label}
\newcommand \ee {\end{equation}}
\newcommand \la \langle
\newcommand \ra \rangle

\newcommand	\RR 		{\mathbb R}

\newcommand \del {{\partial}}

\newcommand \eps \epsilon

\newcommand \wh {\widehat{w}} 

\begin{document}

\title{Stability of a class of semilinear waves in $2+1$ dimension under null condition} 

\author{Shijie Dong\footnote{Sorbonne Universit\'e, CNRS, Laboratoire Jacques-Louis Lions, 4, Place Jussieu, 75252 Paris, France. 
Email: dongs@ljll.math.upmc.fr, shijiedong1991@hotmail.com.
}}

\date{\today}

\maketitle

\begin{abstract} 
We will show that in $\RR^{2+1}$ semilinear wave equations of the form $-\Box u = u Q(\del u; \del u)$ possess global-in-time solutions if the null condition on $Q(\del u; \del u)$ is assumed. As a consequence, we also provide a new proof, after \cite{Wong}, on the small data global solutions to the wave map equation in $\RR^{2+1}$ and no compactness assumptions on the initial data are needed.
\end{abstract}
\maketitle 

\tableofcontents


\

\section{Introduction}
\label{sec:1}

\paragraph{Brief history}

The study of nonlinear wave equations has been an active research field since decades ago, and tremendous results have been obtained in $\RR^{3+1}$.
It is well-known, for example see the examples by John \cite{John2, John}, that wave equations with quadratic nonlinearities might not have global-in-time solutions. Later on, the celebrated breakthrough by Klainerman \cite{Klainerman80, Klainerman852, Klainerman86} relying on the vector field method and Christodoulou \cite{Christodoulou} relying on the conformal method, showed that global-in-time solutions exist for wave equations with null nonlinearities. We also recall other work on the three dimensional wave equations \cite{L-R-cmp, L-R-annals} and \cite{P-S-cpam}, which generalise the notion of null forms. As an application, many physical models, like Dirac equations, Maxwell equations, Einstein equations, are proved to be stable under small perturbations.

Due to the fact that waves in $\RR^{2+1}$ do not decay fast enough, the classical null condition cannot guarantee that semilinear wave equations in $\RR^{2+1}$ with null quadratic nonlinearities have global-in-time solutions.
The existence results of global-in-time solutions to quadratic nonlinear wave equations in $\RR^{2+1}$ was first obtained by Godin \cite{Godin} in the semi linear case and by Alinhac \cite{Alinhac1, Alinhac2} in the quasilinear case, where the author showed that a class of quasilinear wave equations of the form $-\Box u + g^{\alpha \beta \gamma} \del_\gamma u \del_{\alpha \beta} u = 0$ satisfying the null condition are stable under the assumption that the initial data are small and compactly supported.  Later on Zha \cite{Zha} had a thorough study on a large class of wave equations in $\RR^{2+1}$. We also remind one the study of the system of coupled quasilinear wave and Klein-Gordon equations in $\RR^{2+1}$ by Ma \cite{YM1, YM2, YM0}, using the hyperboloidal foliation method \cite{LM0} which dates back to \cite{Klainerman85}. But the compactness assumption on the initial data is needed in all of the above results.

Without the compactness restriction on the initial data, Katayama \cite{Katayama} obtained the global solutions to a class of semilinear wave equations in $\RR^{2+1}$.
Later on, Cai, Lei, and Masmoudi \cite{Cai} removed the compactness assumption on the initial data after \cite{Alinhac1}, and obtained the global well-posedness result for the scalar wave equation which is fully nonlinear.

There is a large literature on the study of wave maps, we are not going to be exhaustive. We only mention the work \cite{Tao2001, Wong} on wave maps in $\RR^{2+1}$.

\paragraph{Model of interest and the main difficulties}

We will consider the following system of semilinear wave equations
\bel{eq:model}
\aligned
- \Box u_i &= R^{jkl}_i u_j Q_0 (u_k, u_l),
\\
u_i (t_0, \cdot) = &u_{i0},
\qquad
\del_t u_i (t_0, \cdot) = u_{i1},
\endaligned
\ee
in which $\Box = \del_\alpha \del^\alpha = -\del_t \del_t + \del_a \del^a$, and $Q_0 (v, w) := - \del_\alpha v \del^\alpha w~ or ~ \del_\alpha v \del_\beta w - \del_\alpha w \del_\beta v$ is the null form.
In the above, $t_0$ is the initial time taken to be 0, and $i, j, k, l \in \{ 1, 2, \cdots, n_0 \}$ with $n_0\geq 1$ the number of unknowns (equations). Greek indices $\alpha, \beta, \cdots$ run in $\{0, 1, 2\}$, Latin letters $a, b, \cdots$ run in $\{1, 2\}$, and the Einstein summation convention is adopted unless specified. Besides, we use $A \lesssim B$ to denote $A \leq C B$ with $C$ a generic constant.

Compared to the existing results in \cite{Alinhac1, Alinhac2, Cai, YM1, YM2}, the main difference is that the nonlinearities in the model problem \eqref{eq:model} include the potential $u_k$ (with no derivatives), instead of only $\del u_k$ and $\del \del u_j$. Recall the standard energy of wave equation is of the form
$$
\sum_{i, \alpha} \int_{\RR^2} | \del_\alpha u_i |^2 \, dx,
$$
which does not include the $L^2$ norm of the potential
$$
\sum_i \int_{\RR^2} | u_i |^2 \, dx.
$$
Thus it requires us to bound $\| u \|_{L^2(\RR^2)}$.

The Hardy inequality 
$$
\Big\| {w \over r} \Big\|_{L^2(\RR^n)}
\leq C \| \nabla w \|_{L^2(\RR^n)}
$$
is only true for $n \geq 3$, and the $L^2$ norm of the potential can be bounded by the conformal energy only when $n \geq 3$, see the remark in \cite{Wong}. The $L^2$ estimates (possibly with weight) for waves in $\RR^{2+1}$ were obtained in \cite{Wong, Liu, YM0}, but the compact support assumption on the initial data is required.

In order to conquer that difficulty and bound $\sum_i \int_{\RR^2} | u_i |^2 \, dx$, we write the wave equations in the Fourier space, and derive the formulation for $u_i$ by solving the ordinary differential equation (see for example \cite{Dong-zeromass}), then a careful treatment on each terms in the formulation gives us the desired result, see Lemma \ref{lem:linear}. However, there is a "loss of derivative" problem occurring when we estimate the highest order energy, see the proof of Proposition \ref{prop:E} for example. We find that this "loss of derivative" problem can be overcome by an observation on the estimates of the null forms and by the aid of the ghost weight energy estimates \cite{Alinhac1}.

\paragraph{Main theorem}

Our goal is to obtain global-in-time solutions to the system \eqref{eq:model}, which is stated now.

\begin{theorem}\label{thm:main}
Consider the system of coupled wave equations \eqref{eq:model}, and let $N \geq 1$ be a sufficiently large integer. The parameters $R^{j k l}_i$ are taken to be constants. Then there exists a small $\eps_0 > 0$, such that the Cauchy problem \eqref{eq:model} admits a global-in-time solution $(u_i)$ as long as the initial data $(u_{i0}, u_{i1})$ satisfy the smallness condition
$$
\sum_{|I| \leq N+1}\| \Lambda^I u_{i0} \|_{L^2(\RR^2)} + \sum_{|J|\leq N}\| \Lambda^J u_{i1} \|_{L^2(\RR^2) \cap L^1(\RR^2)}
< \eps, 
\qquad
\text{for any } \eps \in (0, \eps_0),
$$
with $\Lambda = \del_a, r\del_r, \Omega_{ab}$.
Moreover, the solution decays almost sharply with
\be
|u(t, x)|
\lesssim (1+t+|x|)^{-1/2} (1+|t-|x||)^{-1/2} \log(1+t). 
\ee
\end{theorem}

In general the smallness condition on $\| \Lambda^J u_{i1} \|_{L^1(\RR^2)}$ is not assumed, but it will be used in the proof of Lemma \ref{lem:linear}.

\paragraph{Organisation}

In Section \ref{sec:pre}, we revisit some preliminaries of the wave equations and the vector field method. Next in Section \ref{sec:linear}, we prove the $L^2$ norm of the potential solving the linear wave equation. Then we give the proof to Theorem \ref{thm:main} in Section \ref{sec:proof}.


\section{Preliminaries}\label{sec:pre}

We work in the $(2+1)$ dimensional spacetime with signature $(-, +, +)$. A point in $\RR^{2+1}$ is denoted by $(x_0, x_1, x_2) = (t, x_1, x_2)$, and its spacial radius is denoted by $r = \sqrt{x_1^2 + x_2^2}$. We will use 
\be 
E(w, t)
:=
\int_{\RR^2} \Big(|\del_t w|^2 + \sum_a |\del_a w|^2 \Big) \, dx
\ee
to denote the energy of a sufficiently nice function $w = w(t, x)$ on the constant time slice.

The vector fields we will use in the following analysis include:
\bei
\item Translations: $\del_\alpha$, \quad $\alpha = 0, 1, 2$.

\item Rotations: $\Omega_{ab} = x_a \del_b - x_b \del_a$,  \quad $a, b = 1, 2$.

\item Lorentz boosts: $L_a = x_a \del_t + t \del_a$, \quad $a = 1, 2$.

\item Scaling vector field: $L_0 = t \del_t + r \del_r$.

\eei
We will use $\Gamma$ to denote the vector fields in 
$$
V := \{ \del_\alpha, \Omega_{ab}, L_a, L_0 \}.
$$

The following well-known results of commutators will be also frequently used.

\begin{lemma}
For any $\Gamma', \Gamma'' \in V$ we have
\be 
[\Box, \Gamma'] = C \Box,
\qquad
[\Gamma', \Gamma''] = \sum_{\Gamma \in V, C_\Gamma} C_\Gamma \Gamma,
\ee
with $C, C_\Gamma$ some constants.
\end{lemma}

In order to estimate null forms and to overcome the problem of "loss of derivative", we need the following lemma which gives very detailed estimates on the null forms and can be found in \cite{Sogge} for example.

\begin{lemma}\label{lem:null}
It holds that
\be 
(1+t) |Q_0(v, w)|
\lesssim
\sum_{|I| = 1} | \Gamma^I v | \sum_{a, \alpha} \big( | L_a w | + | \del_\alpha w| \big).
\ee
Besides, if we act the vector field $\Gamma^I$ on the null form $Q_0 (v, w)$, a similar result holds
\be 
\aligned
(1+t) |\Gamma^I Q_0(v, w)|
&\lesssim
\sum_{|I_1| \leq |I|, a, \alpha} \big(| \Gamma^{I_1} v | + |L_a \Gamma^{I_1} v | + |\del_\alpha \Gamma^{I_1} v | \big)  \sum_{1\leq |I_2| \leq |I|/2 + 1} | \Gamma^{I_2} w |
\\
&+ \sum_{ |I_1| \leq |I|, a, \alpha} \big( | \Gamma^{I_1} w | + |L_a \Gamma^{I_1} w | + |\del_\alpha \Gamma^{I_1} w | \big)  \sum_{1\leq |I_2| \leq |I|/2 + 1} | \Gamma^{I_2} v |.
\endaligned
\ee
\end{lemma}

We next recall the Klainerman-Sobolev inequality (see \cite{Sogge} for example) in $\RR^{2+1}$.

\begin{proposition}\label{prop:Sobolev}
It holds that
\be 
\langle t+r \rangle^{1/2} \langle t-r \rangle^{1/2}  |u|
\lesssim
\sum_{|I| \leq 2} \| \Gamma^I u \|_{L^2(\RR^2)},
\ee
with $\langle a \rangle = \sqrt{1 + |a|^2}$.

\end{proposition}

The following energy estimates of the ghost weight method by Alinhac \cite{Alinhac1} will play a vital role in compensating the "loss of derivative" issue.

\begin{proposition}
Let $w$ be the solution to 
$$
- \Box w = f,
$$
then it holds
\be 
\aligned
&E_{gst1} (w, t)
\leq
\int_{\RR^2} e^{q} \big( |\del_t w|^2 + \sum_a |\del_a w|^2  \big) \, dx (0)
+
2 \int_0^t \int_{\RR^2} f \del_t w e^q \, dxdt,
\endaligned
\ee
in which $q = \arctan (r-t)$, and 
\be 
E_{gst1} (w, t)
=
\int_{\RR^2} e^q \big( |\del_t w|^2 + \sum_a |\del_a w|^2  \big) \, dx (t)
+
\sum_{a} \int_0^t \int_{\RR^2} {e^q \over r^2 \langle r-t \rangle^2} \big| (x_a \del_t + r \del_a) w \big|^2 \, dxdt.
\ee
\end{proposition}

Since $-\pi/2 \leq q \leq \pi/2$, we equivalently have
\bel{eq:ghost1} 
\aligned
E_{gst2} (w, t)
\lesssim
\int_{\RR^2} \big( |\del_t w|^2 + \sum_a |\del_a w|^2  \big) \, dx (0)
+
\int_0^t \int_{\RR^2} f \del_t w \, dxdt,
\endaligned
\ee
with
\be 
E_{gst2} (w, t)
=
\int_{\RR^2} \big( |\del_t w|^2 + \sum_a |\del_a w|^2  \big) \, dx (t)
+
\sum_{a} \int_0^t \int_{\RR^2} {1 \over r^2 \langle r-t \rangle^2} \big| (x_a \del_t + r \del_a) w \big|^2 \, dxdt.
\ee


\section{$L^2$ estimates on wave equation}\label{sec:linear}

We have the following lemmas which help bound the $L^2$ norm of the solution (with no derivatives in front) to wave equations.

\begin{lemma}\label{lem:linear}
Let $w$ be the solution to the linear wave equation
\bel{eq:w} 
\aligned
&- \Box w = f,
\\
w(1, \cdot) = &w_0,
\quad
\del_t w(1, \cdot) = w_1.
\endaligned
\ee
We assume that
\be 
\| w_0 \|_{L^2(\RR^2)} + \| w_1 \|_{L^2(\RR^2) \cap L^1(\RR^2)}
< +\infty,
\ee
as well as either
\be 
\| f(t, \cdot) \|_{L^2(\RR^2) \cap L^1(\RR^2)}
< C_f (1+t)^{-1 + \beta},
\qquad
\beta \in (-\infty, 1),
\ee
or
\be 
\int_0^t \| f(t', \cdot) \|_{L^2(\RR^2) \cap L^1(\RR^2)} \, dt'
\lesssim 
C_f (1+t)^\beta,
\qquad
\beta \in [0, 1).
\ee
Then the following $L^2$ norm bound is valid
\begin{eqnarray}\label{eq:l2bound}
\| u \|_{L^2(\RR^2)}
\lesssim
\left\{
\begin{array}{lll}
&
\| w_0 \|_{L^2(\RR^2)}
+
\log^{1/2} (2+t) \| w_1 \|_{L^2(\RR^2) \cap L^1(\RR^2)}
+
\log^{1/2} (2+t) C_f,
\quad
& \beta < 0,
\\
&
\| w_0 \|_{L^2(\RR^2)}
+
\log^{1/2} (2+t) \| w_1 \|_{L^2(\RR^2) \cap L^1(\RR^2)}
+
\log^{3/2} (2+t) C_f,
\quad
& \beta = 0,
\\
&
\| w_0 \|_{L^2(\RR^2)}
+
\log^{1/2} (2+t) \| w_1 \|_{L^2(\RR^2) \cap L^1(\RR^2)}
+
(1+t)^\beta \log^{1/2} (2+t) C_f,
\quad
& 0< \beta <1.
\end{array}
\right.
\end{eqnarray}
\end{lemma}

\begin{proof}

$Step ~1.$ Expressing the solution in Fourier space.

We write the equation \eqref{eq:w} in the Fourier space to get
$$
\aligned
\del_{tt} \wh(t, \xi) + |\xi|^2 \wh(t, \xi) = \fh(t, \xi),
\\
\wh(1, \cdot) = \wh_0,
\qquad
\del_t \wh(1, \cdot) = \wh_1.
\endaligned
$$
We solve the ordinary differential equation in $t$ to arrive at the expression of the solution $w$ in Fourier space
$$
\wh(t, \xi)
=
\cos (t |\xi|) \wh_0
+
{\sin (t |\xi|) \over |\xi|} \wh_1
+
\int_0^t {\sin \big( (t-t') |\xi|\big) \over |\xi|} \fh(t') \, dt'.
$$
Thus the $L^2$ norm of $w$ can be bounded by the following three terms
\be 
\aligned
\|w\|_{L^2(\RR^2)}
&\lesssim
\|w_0\|_{L^2(\RR^2)}
+
\Big\|{\sin (t |\xi|) \over |\xi|} \wh_1 \Big\|_{L^2(\RR^2)}
+
\int_0^t \Big\|{\sin \big( (t-t') |\xi|\big) \over |\xi|} \fh(t') \Big\|_{L^2(\RR^2)} \, dt'
\\
& =: \|w_0\|_{L^2(\RR^2)} + A_1 + A_2.
\endaligned
\ee
The last two terms $A_1, A_2$ needs a more careful treatment.

$Step ~2.$ Estimating the term $A_1$.

We first bound the term $A_1^2$
$$
\aligned
A_1^2
&=
\int_{\{\xi : |\xi| \leq 1\}} {\sin^2 (t |\xi|) \over |\xi|^2} |\wh_1|^2 \, d\xi
+
\int_{\{\xi : |\xi| \geq 1\}} {\sin^2 (t |\xi|) \over |\xi|^2} |\wh_1|^2 \, d\xi
\\
&\leq 
|| \wh_1 ||^2_{L^\infty(\RR^2)} \int_{\{\xi : |\xi| \leq 1\}} {\sin^2 (t |\xi|) \over |\xi|^2} \, d\xi
+
\| \wh_1 \|_{L^2(\RR^2)}^2.
\endaligned
$$
Next we proceed by estimating 
$$
\aligned
\int_{\{\xi : |\xi| \leq 1\}} {\sin^2 (t |\xi|) \over |\xi|^2} \, d\xi
&=
\int_{S^1} dS^1 \int_0^{1} {\sin^2 (t |\xi|) \over |\xi|} \, d|\xi|
\\
&\lesssim
\int_0^{t} {\sin^2 p \over p} \, d p
\\
&\lesssim
\int_0^{1} 1 \, d p
+
\int_{1}^{t+2} {1 \over p} \, d p,
\endaligned
$$
where we used the simple fact that $\sin |p| \leq |p|$ and $\sin p \leq 1$.
By gathering the above results, we obtain
$$
A_1^2
\lesssim
\log(2+t) \| w_1 \|_{L^2(\RR^2) \cap L^1(\RR^2)}^2,
$$
which gives us
\be 
A_1
\lesssim
\log^{1/2}(2+t) \| w_1 \|_{L^2(\RR^2) \cap L^1(\RR^2)}.
\ee

$Step~3.$ Estimating the term $A_2$.

By the analysis in $Step~2$, we have
$$
\aligned
\Big\|{\sin \big( (t-t') |\xi|\big) \over |\xi|} \fh(t') \Big\|_{L^2(\RR^2)}
&\lesssim
\log^{1/2}(2 + t- t') \| f(t') \|_{L^2(\RR^2) \cap L^1(\RR^2)}
\\
&\leq 
C_f \log^{1/2}(2+t - t') \, (1+t')^{-1 + \beta},
\endaligned
$$
in which we used the decay assumption on $ \| f(t') \|_{L^2(\RR^2) \cap L^1(\RR^2)}$ in the last step.
Hence we have
$$
A_2
\lesssim
\log^{1/2} (2+t) \int_0^t (1+t')^{-1 + \beta} \, dt'.
$$
By calculating the integral in terms of the values of $\beta$, and gathering the results in the previous steps, we thus complete the proof.

\end{proof}

This lemma is of vital importance in the proof of Theorem \ref{thm:main}. Recall that the $L^2$ norm estimates on solutions to wave equations in $\RR^{2+1}$ have been obtained before \cite{Wong, Liu, YM0}, but the compactness assumption on the initial data is needed (although the compactness assumption can be removed in the theorem in \cite{Wong}). To the best of our knowledge, the estimates in Lemma \ref{lem:linear} is the first such result where no compactness assumptions are imposed on the initial data.

In Lemma \ref{lem:linear} there is a $\log$ or polynomial growth in the bounds of the $L^2$ norms. As far as we understand, such growth also exists in \cite{Wong, Liu, YM0}.

\begin{remark}
Recall that by using the conformal energy estimates in $\RR^{d+1}$ with $d \geq 3$, the energy 
$$
\sum_{|I| \leq 1} \| \Gamma^I w \|_{L^2(\RR^d)}
$$
can be bounded by the conformal energy.
However we note that in Lemma \ref{lem:null}, we only have the upper bound for 
$$
\| w \|_{L^2(\RR^2)},
$$
which means we "lose" one order of derivative, and that is what we interpret as the issue of "loss of derivative". 
\end{remark}


\section{Proof of the main theorem}\label{sec:proof}


Relying on a standard bootstrap argument, we are going to give the proof of Theorem \ref{thm:main}.

According to the smallness of the initial data, we can assume the following bounds hold in the time interval $[t_0, T)$ with $T> t_0$
\bel{eq:BA}
\aligned
\sum_i E_{gst2} ( \Gamma^I u_i, t)^{1/2}
&\leq C_1 \eps,
\qquad
|I| \leq N, 
\\
\sum_i \big\| \Gamma^I u_i \big\|_{L^2(\RR^2)}
&\leq C_1 \eps \log(1+t),
\qquad
|I| \leq N-1,
\\
\sum_i \big\| \Gamma^I u_i \big\|_{L^2(\RR^2)}
&\leq C_1 \eps (1+t)^{1/4 + 2\delta},
\qquad
|I| = N.
\endaligned
\ee
In the above, $\delta>0$ is some small constant, $C_1$ is some large constant (to ensure $T>t_0$) which is to be determined, and $T$ is defined by
\bel{eq:T-def} 
T := \{ t > t_0 : \eqref{eq:BA} ~holds\}.
\ee

Our goal is to show the following refined estimates
\bel{eq:BA-refined}
\aligned
\sum_i E_{gst2} ( \Gamma^I u_i, t)^{1/2}
&\leq {1\over 2} C_1 \eps,
\qquad
|I| \leq N, 
\\
\sum_i \big\| \Gamma^I u_i \big\|_{L^2(\RR^2)}
&\leq {1\over 2} C_1 \eps \log(1+t),
\qquad
|I| \leq N-1,
\\
\sum_i \big\| \Gamma^I u_i \big\|_{L^2(\RR^2)}
&\leq {1\over 2} C_1 \eps (1+t)^{1/4 + 2\delta},
\qquad
|I| = N.
\endaligned
\ee
The time bound $T$ has two kinds of possible values: $+\infty$ or some finite number. If $T = + \infty$, then Theorem \ref{thm:main} is proved. If $T < + \infty$ is some finite number and the refined estimates in \eqref{eq:BA-refined} are also established, then this contradicts the definition of $T$ in \eqref{eq:T-def}, which means $T$ must be $+\infty$. In either case, we are led to $T = +\infty$, which implies Theorem \ref{thm:main}.

A direct result of the Sobolev-Klainerman inequality in Proposition \ref{prop:Sobolev} gives us the following pointwise decay results.

\begin{lemma}\label{lem:sup}
Let the bootstrap assumptions in \eqref{eq:BA} be true, then we have
\be 
\aligned
\langle t+r \rangle^{1/2} \langle t-r \rangle^{1/2} \sum_i | \del \Gamma^J u_i |
&\lesssim C_1 \eps,
\qquad
|J| \leq N-3,
\\
\langle t+r \rangle^{1/2} \langle t-r \rangle^{1/2} \sum_i | \Gamma^J u_i |
&\lesssim C_1 \eps \log(1+t),
\qquad
|J| \leq N-3.
\endaligned
\ee
\end{lemma}

\begin{proposition}[Improved estimates on the energy norm $E$]\label{prop:E}
Under the bootstrap assumptions in \eqref{eq:BA}, we have
\be
\sum_i E_{gst2}( \Gamma^I u_i, t)^{1/2}
\lesssim \eps + (C_1 \eps)^2,
\qquad
|I| \leq N.
\ee

\end{proposition}

\begin{proof}

Recall the system of equations \eqref{eq:model}, and we act the vector filed $\Gamma^I$ with $I \leq N-1$ to get
$$
\aligned
- \Box \Gamma^I u_i 
= \sum_{|I_1| + |I_2| \leq N-1} R^{jkl}_i \Gamma^{I_1} u_j \Gamma^{I_2} Q_0 (u_k, u_l).
\endaligned
$$
By using the ghost weight energy estimate \eqref{eq:ghost1}, we have
$$
\aligned
\sum_i E_{gst2}( \Gamma^I u_i, t)
&\lesssim
\int_{\RR^2} \big( |\del_t \Gamma^I u|^2 + \sum_a |\del_a \Gamma^I u|^2  \big) \, dx (0)
\\
&+
\sum_{|I_1| + |I_2| \leq N-1} \int_0^t \int_{\RR^2} \Big|  R^{jkl}_i \Gamma^{I_1} u_j \Gamma^{I_2} Q_0 (u_k, u_l) \del_t \Gamma^I u_i \Big| \, dxdt'.
\endaligned
$$
We proceed by estimating the last term
$$
\aligned
&\sum_{|I_1| + |I_2| \leq N-1} \int_0^t \int_{\RR^2} \Big|  R^{jkl}_i \Gamma^{I_1} u_j \Gamma^{I_2} Q_0 (u_k, u_l) \del_t \Gamma^I u_i \Big| \, dxdt'
\\
\lesssim
& \sum_{|I_1| + |I_2| + |I_3| \leq N-1} \int_0^t \int_{\RR^2} {1\over 1+t'} |\Gamma^{I_1} u| | \Gamma^{I_2} u| |\Gamma^{I_3} u| | \del_t \Gamma^I u| \, dxdt'
\\
\lesssim 
&\sum_{|I_1| \leq N, |I_2| \leq N-3} \int_0^t {1\over 1+t'} \big\| \Gamma^{I_1} u \big\|_{L^2(\RR^2)}^2 \big\| \Gamma^{I_2} u \big\|_{L^\infty(\RR^2)}^2 \, dt'
\\
\lesssim
& (C_1 \eps)^4 \int_0^t (1+t')^{-5/4} \, dt'
\lesssim
(C_1 \eps)^4.
\endaligned
$$
Thus we easily get 
\be 
\sum_i E_{gst2}( \Gamma^I u_i, t)^{1/2}
\lesssim \eps + (C_1 \eps)^2,
\qquad
|I| \leq N-1.
\ee

Next, we look at the case of $|I| = N$, and we have
$$
\aligned
\sum_i E_{gst2}( \Gamma^I u_i, t)
&\lesssim
\int_{\RR^2} \big( |\del_t \Gamma^I u|^2 + \sum_a |\del_a \Gamma^I u|^2  \big) \, dx (0)
\\
&+
 \sum_{|I_1| + |I_2| \leq N}  \int_0^t \int_{\RR^2} \Big| R^{jkl}_i \Gamma^{I_1} u_j \Gamma^{I_2} Q_0 (u_k, u_l) \del_t \Gamma^I u_i \Big| \, dxdt'.
\endaligned
$$
We first divide the last estimate into two parts
$$
\aligned
&\sum_{|I_1| + |I_2| \leq N} \int_0^t \int_{\RR^2} \Big|  R^{jkl}_i \Gamma^{I_1} u_j \Gamma^{I_2} Q_0 (u_k, u_l) \del_t \Gamma^I u_i \Big| \, dxdt'
\\
\leq
&\sum_{|I_1| + |I_2| \leq N, |I_2| \leq N-1} \int_0^t \int_{\RR^2} \Big|  R^{jkl}_i \Gamma^{I_1} u_j \Gamma^{I_2} Q_0 (u_k, u_l) \del_t \Gamma^I u_i \Big| \, dxdt'
\\
+
&\sum_{ |I_2| = N} \int_0^t \int_{\RR^2} \Big|  R^{jkl}_i u_j \Gamma^{I_2} Q_0 (u_k, u_l) \del_t \Gamma^I u_i \Big| \, dxdt'
:= P_1 + P_2.
\endaligned
$$
We notice that the estimate of $P_1$ part follows from what we have done for the case of $|I| \leq N-1$ with
$$
P_1 \lesssim (C_1 \eps)^4
$$ 
so we only need to estimate $J_2$. 
We split the space domain into two parts for each fixed $t$
$$
S_{int} := \{ x: |x| \leq (1+t)^{9/8} \},
\qquad
S_{ext} := \{ x: |x| \geq (1+t)^{9/8} \},
$$
and the term $P_2$ can be divided into two parts
$$
\aligned
P_2
&= 
P_2(S_{int}) + P_2(S_{ext}),
\\
P_2(S)
&:=
\sum_{ |I_2| = N} \int_0^t \int_{S} \Big|  R^{jkl}_i u_j \Gamma^{I_2} Q_0 (u_k, u_l) \del_t \Gamma^I u_i \Big| \, dxdt'.
\endaligned
$$
In the region $S_{ext}$, it holds $(1+t)^{1/8} \lesssim \langle r-t \rangle$, and we do not even need the null structure to have the bound
$$
\aligned
P_2(S_{ext})
&\lesssim
\sum_{|I_1| \leq N, |I_2| \leq N-3}  \int_0^t \big\| \Gamma^{I_1} u \big\|_{L^2(\RR^2)}^2 \big\| \Gamma^{I_2} u \big\|_{L^\infty(S_{ext})}^2 \, dt'
\\
&\lesssim 
(C_1 \eps)^4 \int_0^t (1+t')^{-9/8} \, dt'
\lesssim (C_1 \eps)^4.
\endaligned
$$

Then we only need to estimate the term $P_2(S_{int})$.
By recalling the estimates for null forms in Lemma \ref{lem:null}, it holds
$$
\aligned
P_2(S_{int})
&\lesssim
\sum_{|I_1| = N, a} \int_0^t \int_{S_{int}} {1\over 1+t'} |u| \big| L_a \Gamma^{I_1} u \big| \big| \Gamma u \big| \big| \del_t \Gamma^I u \big| \, dxdt'
\\
&+ \sum_{|I_1| = N, \alpha}  \int_0^t \int_{S_{int}} {1\over 1+t'} |u|  \big| \del_\alpha \Gamma^{I_1} u \big| \big| \Gamma u \big| \big| \del_t \Gamma^I u \big| \, dxdt'
\\
&+ \sum_{\substack{|I_1| \leq N, |I_2| \leq N\\ |I_1| + |I_2| \leq N+2}} \int_0^t \int_{S_{int}} {1\over 1+t'} |u| \big| \Gamma^{I_1} u \big| \big| \Gamma^{I_2} u \big| \big| \del_t \Gamma^I u \big| \, dxdt'
:= P_{21} + P_{22} + P_{23}.
\endaligned
$$
It is not hard to see that the terms $P_{22}, P_{23}$ can be estimated in the same fashion as the term $P_1$ with
$$
P_{22} + P_{23} \lesssim (C_1 \eps)^4,
$$
so we only focus on the estimate of $P_{21}$.
We decompose the Lorentz boosts $L_a$ as
$$
L_a = x_a \del_t + t \del_a
= \big( x_a \del_t + r \del_a \big) + (t-r) \del_a.
$$
Hence we have
$$
\aligned
P_{21}
&\lesssim
\sum_{|I_1| = N, a} \int_0^t \int_{S_{int}} {1\over 1+t'} |u| \big| ( x_a \del_t + r \del_a) \Gamma^{I_1} u \big| \big| \Gamma u \big| \big| \del_t \Gamma^I u \big| \, dxdt'
\\
&+
\sum_{|I_1| = N, a} \int_0^t \int_{S_{int}} {1\over 1+t'} |u| \big| (t'-r) \del_a \Gamma^{I_1} u \big| \big| \Gamma u \big| \big| \del_t \Gamma^I u \big| \, dxdt'
\\
&\lesssim
\sum_{|I_1| = N, a} \int_0^t \int_{S_{int}} {\big| ( x_a \del_t + r \del_a) \Gamma^{I_1} u \big| \over (1+t') \langle r-t' \rangle } \langle r-t \rangle |u|  \big| \Gamma u \big| \big| \del_t \Gamma^I u \big| \, dxdt'
\\
&+
\sum_{|I_1| = N, a} \int_0^t \int_{S_{int}} \langle r-t' \rangle {|u| \big| \Gamma u \big|  \over 1+t'}  \big| \del_a \Gamma^{I_1} u \big| \big| \del_t \Gamma^I u \big| \, dxdt'
\\
&\lesssim
\sum_{|I_1| = N, a} \int_0^t  \Big\| {\big| ( x_a \del_t + r \del_a) \Gamma^{I_1} u \big| \over (1+t')^{9/8} \langle r-t' \rangle } \Big\|_{L^2(S_{int})} \big\| \del_t \Gamma^I u  \big\|_{L^2(\RR^2)}  (1+t')^{1/8} \big\|  \langle r-t' \rangle |u|  \big| \Gamma u \big|  \big\|_{L^\infty(\RR^2)} \, dt'
\\
&+ \sum_{|I_1| = N, a} \int_0^t \big\| \del_a \Gamma^{I_1} u  \big\|_{L^2(\RR^2)}^2 \Big\| \langle r-t' \rangle {|u| \big| \Gamma u \big|  \over 1+t'} \Big\|_{L^\infty(\RR^2)} \, dt' =: P_{211} + P_{212}.
\endaligned
$$
Successively, we deduce
$$
P_{212}
\lesssim
\sum_{|I_1| = N} \int_0^t (C_1 \eps)^4 (1+t')^{-3/2} \, dt'
\lesssim
(C_1 \eps)^4.
$$
In order to estimate the term $P_{211}$, we note that in the region $S_{int}$, it holds $1/(1+t)^{9/8} \leq 1/ |x|$, and we thus have
$$
\aligned
P_{211}
&\lesssim
\sum_{|I_1| \leq N, a} \int_0^t \Big\| {\big| ( x_a \del_t + r \del_a) \Gamma^{I_1} u \big| \over r \langle r-t' \rangle } \Big\|_{L^2(\RR^2)} (C_1 \eps)^3 (1+t')^{-5/8} \, dt'
\\
&\lesssim
\sum_{|I_1| \leq N, a} (C_1 \eps)^3 \Big( \int_0^t \Big\| {\big| ( x_a \del_t + r \del_a) \Gamma^{I_1} u \big| \over r \langle r-t' \rangle } \Big\|_{L^2(\RR^2)}^2 \, dt' \Big)^{1/2} \Big( \int_0^t (1+t')^{-5/4} \, dt'  \Big)^{1/2}
\\
&\lesssim
(C_1 \eps)^4.
\endaligned
$$

By gathering the estimates above, the proof is complete.

\end{proof}

\begin{proposition}[Improved estimates on the $L^2$ norm]\label{prop:L2}
Under the bootstrap assumptions in \eqref{eq:BA}, the following estimates hold
\bel{eq:L2Improved} 
\aligned
\sum_i \big\| \Gamma^I u_i \big\|_{L^2(\RR^2)}
&\leq \eps \log(1+t) + (C_1 \eps)^2 \log(1+t),
\qquad
|I| \leq N-1,
\\
\sum_i \big\| \Gamma^I u_i \big\|_{L^2(\RR^2)}
&\leq \eps (1+t)^{1/4 + 2\delta} + (C_1 \eps)^2 (1+t)^{1/4 + 2\delta},
\qquad
|I| = N.
\endaligned
\ee
\end{proposition}

\begin{proof}
According to Lemma \ref{lem:linear}, in order to show the first improved estimate in \eqref{eq:L2Improved} it suffices to show for each $i$ that
$$
\Big\| \Gamma^{I} \Big( R^{jkl}_i u_j Q_0 (u_k, u_l) \Big) \Big\|_{L^2 (\RR^2) \cap L^1 (\RR^2)}
\lesssim (C_1 \eps)^2 (1+t)^{-9/8},
$$
for all $|I| \leq N-1$.
It is easy to see that
$$
\aligned
\Big\| \Gamma^I \big( R^{jkl}_i u_j Q_0 (u_k, u_l) \big) \Big\|_{L^2(\RR^2) \cap L^1(\RR^2)}
&\lesssim
{1\over 1 + t} \sum_{i, |I_1| \leq N } \big\| \Gamma^{I_1} u_i \big\|_{L^2(\RR^2)}^2  \sum_{|I_2| \leq N - 3}  \| \Gamma^{I_2} u_i \|_{L^\infty(\RR^2)}
\\
&\lesssim (C_1 \eps)^2 t^{-5/4 + 4\delta}
\leq (C_1 \eps)^2 (1+t)^{-9/8},
\endaligned
$$
where we have used Lemma \ref{lem:null}.

Next we consider the case of $|I| = N$, and the proof in Proposition \ref{prop:E} indicates that it is easy to have
$$
\sum_{|I_1| + |I_2| \leq N, |I_2| \leq N-1} \Big\|  \Big( R^{jkl}_i \Gamma^{I_1} u_j \Gamma^{I_2} Q_0 (u_k, u_l) \Big) \Big\|_{L^2 (\RR^2) \cap L^1 (\RR^2)}
\lesssim (C_1 \eps)^3 (1+t)^{-5/4},
$$
which is integrable.
Then in order to estimate the rest term
$$
\int_0^t \Big\|  \Big( R^{jkl}_i u_j \Gamma^{I} Q_0 (u_k, u_l) \Big) \Big\|_{L^2 (\RR^2) \cap L^1 (\RR^2)} \, dt',
$$
we split the space region into two parts
$$
S_1 := \{ x: \langle t - r \rangle \leq (1+t)^{1/2} \},
\qquad
S_2 := \{ x: \langle t - r \rangle \geq (1+t)^{1/2} \}.
$$
We easily have in the region $S_2$ that
$$
\aligned
&\int_0^t \Big\|  R^{jkl}_i u_j \Gamma^{I} Q_0 (u_k, u_l) \Big\|_{L^2 (S_2) \cap L^1 (S_2)} \, dt'
\\
\lesssim
&\sum_{|I_1| \leq N, \alpha} \int_0^t \| \del_\alpha \Gamma^{I_1} u \|^2_{L^2(\RR^2)} \| u \|_{L^\infty(S_2)} \, dt'
\\
\lesssim
&
(C_1 \eps)^3 \int_0^t (1+t')^{-3/4} \log(1+t') \, dt'
\lesssim (C_1 \eps)^3 (1+t)^{1/4} \log(1+t).
\endaligned
$$
We next estimate
$$
\aligned
&\int_0^t \Big\| R^{jkl}_i u_j \Gamma^{I} Q_0 (u_k, u_l) \Big\|_{L^2 (S_1) \cap L^1 (S_1)} \, dt'
\\
\lesssim
&\sum_{|I_1| \leq N} \int_0^t {|u| \over 1+t'} \big\| \Gamma^{I_1} u \big\|_{L^2(\RR^2)}^2 \, dt'
+
\sum_{|I_1| = N, a} \int_0^t \Big\| {|u| \over 1+t'} \big(|L_a \Gamma^{I_1} u|  \big) \big| \Gamma u \big|  \Big\|_{L^1 (S_1)} \, dt'
\\
 &+
\sum_{|I_1| = N, \alpha} \int_0^t \Big\| {|u| \over 1+t'} \big( |\del_\alpha \Gamma^{I_1} u| \big) \big| \Gamma u \big| \Big\|_{L^1 (S_1)} \, dt'
\\
&\lesssim
(C_1 \eps)^3
+
\sum_{|I_1| = N, a} \int_0^t \Big\| {|u| \over 1+t'} \big(|(t'\del_a + r \del_t) \Gamma^{I_1} u|  \big) \big| \Gamma u \big|  \Big\|_{L^1 (S_1)} \, dt'
\\
&+
\sum_{|I_1| = N, \alpha} \int_0^t \Big\| {|u| \over 1+t'} \big( \langle r-t' \rangle |\del_\alpha \Gamma^{I_1} u| \big) \big| \Gamma u \big| \Big\|_{L^1 (S_1)} \, dt'.
\endaligned
$$
Since it holds $\langle r-t \rangle \leq (1+t)^{1/2}$ in the region $S_1$, we obtain
$$
\sum_{|I_1| = N, \alpha} \int_0^t \Big\| {|u| \over 1+t'} \langle r-t' \rangle |\del_\alpha \Gamma^{I_1} u| \big| \Gamma u \big| \Big\|_{L^1 (S_1)} \, dt'
\lesssim
(C_1 \eps)^3. 
$$
We are left with the estimate
$$
\aligned
&\sum_{|I_1| = N, a} \int_0^t \Big\| {|u| \over 1+t'} \big(|(t'\del_a + r \del_t) \Gamma^{I_1} u|  \big) \big| \Gamma u \big|  \Big\|_{L^1 (S_1)} \, dt'
\\
=
&\sum_{|I_1| = N, a} \int_0^t \Big\| {|(t'\del_a + r \del_t) \Gamma^{I_1} u| \over \langle r-t' \rangle (1+t')} \langle r-t' \rangle |u| \big| \Gamma u \big|  \Big\|_{L^1 (S_1)} \, dt'
\\
\lesssim
& \sum_{|I_1| = N, a} \int_0^t \Big\| {(t'\del_a + r \del_t) \Gamma^{I_1} u \over \langle r-t' \rangle r} \Big\|_{L^2(\RR^2)} (C_1 \eps) \log(1+t') \big\| \langle  r-t' \rangle u \big\|_{L^\infty (S_1)} \, dt'
\\
\lesssim 
& \sum_{|I_1| = N, a} \int_0^t \Big\| {(t'\del_a + r \del_t) \Gamma^{I_1} u \over \langle r-t' \rangle r} \Big\|_{L^2(\RR^2)} (C_1 \eps)  (1+t')^{-1/4 + \delta} \, dt',
\endaligned
$$
where we have applied the fact that $r \leq 2(1+t)$ within the region $S_1$ in the second step, and the simple relation $\log^2(1+t) \lesssim (1+t)^\delta$ for $t, \delta > 0$ in the last step, and then Cauchy-Schwarz inequality gives us 
$$
\sum_{|I_1| = N, a} \int_0^t \Big\| {|u| \over 1+t'} \big(|(t\del_a + r \del_t) \Gamma^{I_1} u|  \big) \big| \Gamma u \big|  \Big\|_{L^1 (S_1)} \, dt
\lesssim
(C_1 \eps)^3 (1+t)^{1/4 + \delta}.
$$

The proof is done.

\end{proof}

With Proposition \ref{prop:E} and Proposition \ref{prop:L2} prepared, we are now ready to give the proof of Theorem \ref{thm:main}.

\begin{proof}[Proof of Theorem \ref{thm:main}]
Recall the refined estimates in proposition \ref{prop:E} and \ref{prop:L2}, and if we choose $C_1$ very large (say three times larger than the implicit constant in $\lesssim$), and $\eps$ small enough such that $(C_1 \eps)^2 \leq \eps$, then we obtain the refined estimates in \eqref{eq:BA-refined}. Then we know $T>t_0$ cannot be finite, that is $T$ must be $+\infty$, and Theorem \ref{thm:main} is thus verified.

\end{proof}


\section*{Acknowledgments} This project has received funding from the European Research Council (ERC) under the European Union's Horizon 2020 research and innovation program Grant agreement No 637653, project BLOC "Mathematical Study of Boundary Layers in Oceanic Motion". The author was partially supported by the Innovative Training Networks (ITN) grant 642768 entitled ModCompShock. The author would like to thank Allen Fang (Sorbonne University) for letting the author know the reference \cite{Liu} and helpful discussions, and to thank Siyuan Ma (Sorbonne University) for helpful discussions. The author also owes great thanks to Dongbing Zha (Donghua University), who informed the author the references \cite{Godin, Katayama, Zha}.



\end{document}